\documentclass[12pt,a4paper]{article}
\usepackage{amssymb}
\usepackage{amsmath}
\usepackage{amsfonts}

\newtheorem{theorem}{Theorem}[section]

\newtheorem{corollary}[theorem]{Corollary}

\newtheorem{definition}[theorem]{Definition}
\newtheorem{example}[theorem]{Example}

\newtheorem{lemma}[theorem]{Lemma}

\newtheorem{proposition}[theorem]{Proposition}
\newtheorem{remark}[theorem]{Remark}

\newenvironment{proof}[1][Proof]{\noindent\textbf{#1.} }{\ \rule{0.5em}{0.5em}}
\input{tcilatex}
\begin{document}

\title{A Fuglede type theorem for Fourier multiplier operators}
\author{Ben de Pagter and Werner J. Ricker\bigskip \and \textit{Dedicated in
admiration to Jaap Korevaar } \and \textit{on the occasion of his 100-th
birthday}}
\date{}
\maketitle

\begin{abstract}
Let $E$ be a translation invariant Banach function space over an infinite
compact abelian group $G$ and $M_{\varphi }$ be a Fourier multiplier
operator (with symbol $\varphi $) acting on E. It is assumed that $E$ has
order continuous norm and that $E$ is reflection invariant (which ensures
that $\bar{\varphi}$ is also a multiplier symbol for $E$). The following
Fuglede type theorem is established. Whenever $T$ is a bounded linear
operator on $E$ satisfying $M_{\varphi }T=TM_{\varphi }$, then also $M_{\bar{%
\varphi}}T=TM_{\bar{\varphi}}$.
\end{abstract}

\renewcommand{\thefootnote}{}

\renewcommand{\thefootnote}{}

\section{Introduction}

A classical result of B. Fuglede states if $A$ is a bounded normal operator
on a Hilbert space $H$ and $T$ is a bounded linear operator on $H$
satisfying $TA=AT$, then $TA^{\ast }=A^{\ast }T$, \cite{F}. More generally,
C.R. Putnam showed if $A$ and $B$ are bounded normal operators on $H$ such
that $AT=TB$, then $A^{\ast }T=TB^{\ast }$, \cite{P}. This more general
version actually follows from the original Fuglede theorem, \cite{B}.
Various different proofs of Fuglede's theorem are known; see, for example, 
\cite{DS}, p. 934, \cite{Ha1}, \cite{RR}, \cite{R}.

An abstraction of the Fuglede theorem could be as follows. Let $X$ be a
Banach space and $A$ belong to a subalgebra of the bounded linear operators
on $X$ which has an involution $\sharp $. Does it follows that $TA^{\sharp
}=A^{\sharp }T$ whenever $T$ is a bounded linear operator on $X$ satisfying $%
TA=AT$? For $X=L^{p}\left( \mathbb{T}\right) $, with $1\leq p<\infty $ and $%
\mathbb{T}$ the circle group, let $\mathcal{M}_{p}\left( \mathbb{T}\right) $
denote the algebra of all bounded functions $\varphi :\mathbb{Z}\rightarrow 
\mathbb{C}$ having the property that, for every $f\in L^{p}\left( \mathbb{T}%
\right) $ there exists $g\in L^{p}\left( \mathbb{T}\right) $, necessarily
unique, such that its Fourier transform satisfies $\hat{g}=\varphi \hat{f}$
on $\mathbb{Z}$. Denoting $g$ by $M_{\varphi }^{\left( p\right) }f$, the
bounded linear operator $M_{\varphi }^{\left( p\right) }$ so generated on $%
L^{p}\left( \mathbb{T}\right) $ is called the $p$-multiplier operator
corresponding to $\varphi $. It is well known that the complex conjugate $%
\bar{\varphi}$ of $\varphi $ belongs to $\mathcal{M}_{p}\left( \mathbb{T}%
\right) $ whenever $\varphi \in \mathcal{M}_{p}\left( \mathbb{T}\right) $.
Clearly $M_{\varphi }^{\left( p\right) }\longmapsto M_{\bar{\varphi}%
}^{\left( p\right) }$ is an involution on the algebra of all $p$-multiplier
operators on $L^{p}\left( \mathbb{T}\right) $. In this setting the "Fuglede
theorem" is indeed valid, \cite{MR}.

The aim of this note is to give a far reaching generalization of the above
result beyond the setting of $L^{p}$-spaces. The group $\mathbb{T}$ can be
replaced with any infinite compact abelian group $G$ and the $L^{p}$-spaces
can be replaced by the significantly larger class of \textit{translation
invariant Banach function spaces} $E$ over $G$. However, for this class of
spaces $E$ three phenomena arise which do not occur for $L^{p}$-spaces over $%
G$. First, the space $E$ need \textit{not} contain $L^{\infty }\left(
G\right) $. Second, the reflection of a function in $E$ need not belong to $%
E $ (which has the effect that $\varphi \in \mathcal{M}_{E}\left( G\right) $
need not imply that $\bar{\varphi}$ belongs to $\mathcal{M}_{E}\left(
G\right) $) and third, the space of all continuous functions $C\left(
G\right) $ on $G$, even if $C\left( G\right) \subseteq E$, need not be dense
in $E$. However, for the very large and important class of spaces $E$ which 
\textit{do} have all three of these properties it is established in Theorem %
\ref{Thm01} that the "Fuglede theorem" does hold.

Since the methods used rely heavily of the fact that the spaces $E$ are all 
\textit{Banach function spaces} over $G$ (in particular, Banach lattices)
and that certain non-trivial properties of $E$ arising from translation
invariance have to be developed and established, we need to include some
sections to address these features. An attempt has been made to keep the
note as self-contained as possible.

\section{Preliminaries}

In this section we collect together some definitions and facts concerning
Banach function spaces. Let $\left( \Omega ,\Sigma ,\mu \right) $ be a
measure space. Since we are interested in the case where $\Omega =G$ is an
infinite compact abelian group with normalized Haar measure $\mu $, we will
assume that $\mu \left( \Omega \right) =1$ and that $\mu $ is atomless. We
denote by $L^{0}\left( \mu \right) $ the space of all $\mathbb{C}$-valued, $%
\Sigma $-measurable functions (with the usual identification of functions
which are equal $\mu $-a.e. on $\Omega $). Then $L^{0}\left( \mu \right) $
is a Dedekind complete complex Riesz space (or, vector lattice). Recall that
a subset $A\subseteq L^{0}\left( \mu \right) $ is called an (order) \textit{%
ideal} whenever $A$ is a linear subspace of $L^{0}\left( \mu \right) $ with
the property that $\left\vert f\right\vert \leq \left\vert g\right\vert $,
with $f\in L^{0}\left( \mu \right) $ and $g\in A$, implies that $f\in A$.

\begin{definition}
A \emph{Banach function space} (briefly B.f.s.) over $\left( \Omega ,\Sigma
,\mu \right) $ is a Banach space $\left( E,\left\Vert \cdot \right\Vert
_{E}\right) $, where $E\subseteq L^{0}\left( \mu \right) $ is an ideal and $%
\left\Vert \cdot \right\Vert _{E}$ is an absolutely monotone norm on $E$
(that is, $\left\Vert f\right\Vert _{E}\leq \left\Vert g\right\Vert _{E}$
whenever $\left\vert f\right\vert \leq \left\vert g\right\vert $ in $E$).
\end{definition}

Evidently, B.f.s.' are complex (Dedekind complete) Banach lattices.
Therefore, we can freely apply the general theory of Banach lattices to
B.f.s.' For the theory of Banach lattices we refer the reader to \cite{AB2}, 
\cite{MN}, \cite{Sch}, \cite{Z}, and for the theory of B.f.s.' to \cite{Za1}
(Chapter 15), \cite{BS}. The class of B.f.s.' includes the $L^{p}$-spaces ($%
1\leq p\leq \infty $), Orlicz spaces, Lorentz spaces, Marcinkiewicz spaces
and many more. If $E$ and $F$ are two B.f.s.' over $\left( \Omega ,\Sigma
,\mu \right) $ satisfying $E\subseteq F$, then the natural embedding of $E$
into $F$ is continuous. Indeed, the embedding operator is linear and
positive, and a positive linear operator on a Banach lattice is
automatically continuous (see e.g. Theorem 83.12 in \cite{Z}).

\begin{remark}
In Chapter 15 of \cite{Za1} and in the series of papers \cite{LZ}, B.f.s.'
are introduced via function norms. We briefly recall this approach. Denote
by $M\left( \mu \right) $ the set of all (equivalence classes of) extended $%
\mathbb{C}$-valued measurable functions on $\Omega $ (that is, every $f\in
M\left( \mu \right) $ is of the form $f=g+ih$ with $g,h:\Omega \rightarrow %
\left[ -\infty ,\infty \right] $ measurable). The set of all $\left[
0,\infty \right] $-valued functions in $M\left( \mu \right) $ is denoted by $%
M\left( \mu \right) ^{+}$. A \emph{function norm} $\rho $ is defined to be a
map $\rho :M\left( \mu \right) ^{+}\rightarrow \left[ 0,\infty \right] $
satisfying:

\begin{enumerate}
\item[(i)] If $u\in M\left( \mu \right) ^{+}$ and $\rho \left( u\right) =0$,
then $u=0$;

\item[(ii)] $\rho \left( \lambda u\right) =\lambda \rho \left( u\right) $
for all $\lambda \in \mathbb{R}^{+}$ and $u\in M\left( \mu \right) ^{+}$;

\item[(iii)] $\rho \left( u+v\right) \leq \rho \left( u\right) +\rho \left(
v\right) $ for all $u,v\in M\left( \mu \right) ^{+}$;

\item[(iv)] $\rho \left( u\right) \leq \rho \left( v\right) $ whenever $%
u\leq v$ in $M\left( \mu \right) ^{+}$.
\end{enumerate}

\noindent If $u\in M\left( u\right) ^{+}$ and $\rho \left( u\right) <\infty $%
, then $u\left( x\right) <\infty $ for $\mu $-a.e. $x\in \Omega $ (i.e., $%
u\in L^{0}\left( \mu \right) ^{+}$); see \cite{Za1}, Section 63, Theorem 1.
Define 
\begin{equation}
E_{\rho }=\left\{ f\in M\left( \mu \right) :\rho \left( \left\vert
f\right\vert \right) <\infty \right\} .  \label{eq0109}
\end{equation}%
Then $E_{\rho }$ is an ideal in $L^{0}\left( \mu \right) $. Setting $%
\left\Vert f\right\Vert _{E_{\rho }}=\rho \left( \left\vert f\right\vert
\right) $ for all $f\in E_{\rho }$, the space $\left( E_{\rho },\left\Vert
\cdot \right\Vert _{E_{\rho }}\right) $ is a normed space. If $E_{\rho }$ is
complete, then it is a B.f.s.

Conversely, if $\left( E,\left\Vert \cdot \right\Vert _{E}\right) $ is a
B.f.s. and we set 
\begin{equation*}
\rho \left( u\right) =\left\{ 
\begin{array}{ccc}
\left\Vert u\right\Vert _{E} & \text{if} & u\in E^{+} \\ 
\infty & \text{if} & u\notin E^{+}%
\end{array}%
\right. ,\ \ \ u\in M\left( \mu \right) ^{+},
\end{equation*}%
then $\rho $ is a function norm and the space $E$ is recovered via (\ref%
{eq0109}). This shows that the approach in \cite{Za1} is equivalent to ours
and so, all results from \cite{Za1} may be used in our setting (see also 
\cite{Z}, Section 112).
\end{remark}

Given a B.f.s. $\left( E,\left\Vert \cdot \right\Vert _{E}\right) $ over $%
\left( \Omega ,\Sigma ,\mu \right) $, there exists a set $A_{0}\in \Sigma $
with the property that $f=0$ $\mu $-a.e. on $\Omega \backslash A_{0}$ for
each $f\in E$ and, for every $A\in \Sigma $ satisfying $A\subseteq A_{0}$
and $\mu \left( A\right) >0$, there exists $B\in \Sigma $ such that $%
B\subseteq A$, $\mu \left( B\right) >0$ and $\chi _{B}\in E$ (see see \cite%
{Za1}, Section 72 or \cite{Z}, Section 86). The set $A_{0}$ is uniquely
determined (modulo $\mu $-null sets) and is called the \textit{carrier} of $%
E $. Moreover, there exists a sequence $\left( A_{n}\right) _{n=1}^{\infty }$
in $\Sigma $ such that $A_{n}\subseteq A_{0}$ and $\chi _{A_{n}}\in E$ for
all $n$, and $A_{n}\uparrow _{n}A_{0}$. Any B.f.s. that we encounter in this
note has its carrier equal to $\Omega $. In general, if $E\neq \left\{
0\right\} $, then one can replace, without loss of generality, $\Omega $ by $%
A_{0}$.

Recall that a B.f.s. $\left( E,\left\Vert \cdot \right\Vert _{E}\right) $
over $\left( \Omega ,\Sigma ,\mu \right) $ has \textit{order continuous norm}
(briefly o.c.-norm) if $\left\Vert f_{\tau }\right\Vert _{E}\downarrow
_{\tau }0$ whenever $f_{\tau }\downarrow _{\tau }0$ in $E$ (that is, $\left(
f_{\tau }\right) $ is a downwards directed net in the positive cone $E^{+}$
with infimum $0$). Since $E$ is Dedekind complete, it follows from Theorem
103.6 in \cite{Z} that $E$ has o.c.-norm if and only if $E$ has $\sigma $%
-o.c.-norm (that is, $\left\Vert f_{n}\right\Vert _{E}\downarrow _{n}0$ for
any sequence $\left( f_{n}\right) _{n=1}^{\infty }$ in $E^{+}$ satisfying $%
f_{n}\downarrow _{n}0$). For a sequence $\left( f_{n}\right) _{n=1}^{\infty
} $ in $E^{+}$ the condition $f_{n}\downarrow _{n}0$ is equivalent with $%
f_{n}\left( x\right) \downarrow 0$ for $\mu $-a.e. $x\in \Omega $. By way of
example, for $1\leq p<\infty $ the space $E=L^{p}\left( \mu \right) $ has
o.c.-norm but, $L^{\infty }\left( \mu \right) $ does not have o.c.-norm.

Let $\left( E,\left\Vert \cdot \right\Vert _{E}\right) $ be a B.f.s. over $%
\left( \Omega ,\Sigma ,\mu \right) $ with carrier $\Omega $. The \textit{K%
\"{o}the dual} (or \textit{associate space}) $E^{\times }$ of $E$ is defined
by 
\begin{equation*}
E^{\times }=\left\{ g\in L^{0}\left( \mu \right) :\left\Vert g\right\Vert
_{E^{\times }}<\infty \right\} ,
\end{equation*}%
where 
\begin{equation}
\left\Vert g\right\Vert _{E^{\times }}=\sup \left\{ \int_{\Omega }\left\vert
fg\right\vert d\mu :f\in E,\left\Vert f\right\Vert _{E}\leq 1\right\} ,\ \ \
g\in L^{0}\left( \mu \right) .  \label{eq01}
\end{equation}%
Then $\left( E^{\times },\left\Vert \cdot \right\Vert _{E^{\times }}\right) $
is a B.f.s. over $\left( \Omega ,\Sigma ,\mu \right) $ with carrier equal to 
$\Omega $; see \cite{Za1}, Sections 68, 69 and also \cite{Z}, p. 418. From
the definition it is also clear that \textit{H\"{o}lder's inequality} holds,
that is, 
\begin{equation}
\left\vert \int_{\Omega }fg\;d\mu \right\vert \leq \int_{\Omega }\left\vert
fg\right\vert d\mu \leq \left\Vert f\right\Vert _{E}\left\Vert g\right\Vert
_{E^{\times }},\ \ \ f\in E,\ \ g\in E^{\times }.  \label{eq02}
\end{equation}%
According to \cite{Za1}, Section 69, Theorem 1, we also have 
\begin{equation}
\left\Vert g\right\Vert _{E^{\times }}=\sup \left\{ \left\vert \int_{\Omega
}fg\;d\mu \right\vert :f\in E,\left\Vert f\right\Vert _{E}\leq 1\right\} ,\
\ \ g\in E^{\times }.  \label{eq06}
\end{equation}

For each $g\in E^{\times }$, define the linear functional $\varphi
_{g}:E\rightarrow \mathbb{C}$ by setting 
\begin{equation*}
\varphi _{g}\left( f\right) =\int_{\Omega }fg\;d\mu ,\ \ \ f\in E.
\end{equation*}%
It is easy to see that $\varphi _{g}\in E^{\ast }$ (the dual Banach space of 
$E$) and that $\left\Vert \varphi _{g}\right\Vert _{E^{\ast }}=\left\Vert
g\right\Vert _{E^{\times }}$ for all $g\in E^{\times }$ (see (\ref{eq06})).
Consequently, the map $g\longmapsto \varphi _{g}$, for $g\in E^{\times }$,
is an isometric isomorphism from $E^{\times }$ into $E^{\ast }$. Via this
map, $E^{\times }$ may be identified with a closed subspace of $E^{\ast }$
and we frequently denote the \textit{dual pairing} of $E$ and $E^{\times }$
by $\left\langle \cdot ,\cdot \right\rangle $, that is, 
\begin{equation}
\left\langle f,g\right\rangle =\int_{\Omega }fg\;d\mu ,\ \ \ f\in E,\ \ g\in
E^{\times }.  \label{eq0401}
\end{equation}%
Since the carrier of $E^{\times }$ is equal to $\Omega $ (whenever the
carrier of $E$ is $\Omega $), it follows that $E^{\times }$ separates the
points of $E$ (but, in general, $E^{\times }$ is not norming).

The functionals in $E^{\ast }$ which are of the form $\varphi _{g}$ for some 
$g\in E^{\times }$ may be identified with the band $E_{n}^{\ast }$ in $%
E^{\ast }$ of all order continuous functionals on $E$; see \cite{Za1},
Section 69, Theorem 3, and also \cite{Z}, Section 112. In particular, $%
E^{\ast }=E^{\times }$ if and only if $E$ has o.c.-norm (cf. Theorem 2.4.2
in \cite{MN}).

A B.f.s. $E$ over $\left( \Omega ,\Sigma ,\mu \right) $ has the \textit{%
Fatou property} if, for any net $0\leq u_{\alpha }\uparrow _{\alpha }$ in $E$
satisfying $\sup_{\alpha }\left\Vert u_{\alpha }\right\Vert _{E}<\infty $,
there exists $u\in E$ such that $0\leq u_{\alpha }\uparrow _{\alpha }u$ and $%
\left\Vert u_{\alpha }\right\Vert _{E}\uparrow _{\alpha }\left\Vert
u\right\Vert _{E}$. In this case, sequences also suffice; cf. \cite{Z},
Theorem 113.2. By a theorem of Halperin and Luxemburg (see \cite{Za1},
Section 71, Theorem 1), a B.f.s. $E$ has the Fatou property if and only if $%
E=E^{\times \times }$ isometrically. We point out that in \cite{BS} the
Fatou property is included in the definition of a B.f.s.

\section{Translation invariant Banach function \protect\linebreak spaces}

Let $G$ be an infinite compact abelian group and denote by $\mu $ normalized
Haar measure on the Borel $\sigma $-algebra $\mathcal{B}\left( G\right) $.
The space of all regular, complex Borel measures in $G$ is denoted by $%
M\left( G\right) $. The group operation is denoted by $+$ and the identity
element of $G$ is denoted by $0$. Moreover, $\mu \left( -A\right) =\mu
\left( A\right) $ for all $A\in \mathcal{B}\left( G\right) $. For the
general theory of locally compact abelian groups we refer the reader, for
instance, to \cite{La}, \cite{Ru1}. We denote $L^{0}\left( \mu \right) $ by
the more traditional notation $L^{0}\left( G\right) $ and the corresponding $%
L^{p}$-spaces by $L^{p}\left( G\right) $, $1\leq p\leq \infty $. A B.f.s.
over $\left( G,\mathcal{B}\left( G\right) ,\mu \right) $ is simply called a
B.f.s. over $G$. For each $y\in G$, define the \textit{translation operator} 
$\tau _{y}:L^{0}\left( G\right) \rightarrow L^{0}\left( G\right) $ by setting%
\begin{equation*}
\left( \tau _{y}f\right) \left( x\right) =f\left( x-y\right) ,\ \ \ x\in G,
\end{equation*}%
for all $f\in L^{0}\left( G\right) $.

\begin{definition}
\label{Def1101}A \emph{translation invariant B.f.s} over $G$ is a B.f.s. $%
E\subseteq L^{0}\left( G\right) $ such that $\tau _{y}f\in E$ with $%
\left\Vert \tau _{y}f\right\Vert _{E}=\left\Vert f\right\Vert _{E}$ for all $%
y\in G$, whenever $f\in E$.
\end{definition}

Such B.f.s.' include all spaces $L^{p}\left( G\right) $, $1\leq p\leq \infty 
$, all Orlicz spaces, Lorentz spaces and Marcinkiewicz spaces over $G$.
Actually, any rearrangement invariant B.f.s. over $G$ is translation
invariant. However, there exist plenty of translation invariant B.f.s.'
which are \textit{not} rearrangement invariant. We present an example (see
also Example \ref{Ex01} below).

\begin{example}
\label{Ex02}Let $G_{1}$ and $G_{2}$ be two infinite compact abelian groups
(with normalized Haar measure $\mu _{1}$ and $\mu _{2}$, respectively) and
consider the product group $G=G_{1}\times G_{2}$ (with normalized Haar
measure $\mu =\mu _{1}\times \mu _{2}$). Let $1<p,q<\infty $. Define the
function norm $\left\Vert \cdot \right\Vert _{p\times q}$ on $L^{0}\left(
G\right) $ by 
\begin{equation*}
\left\Vert f\right\Vert _{p\times q}=\left( \int_{G_{1}}\left(
\int_{G_{2}}\left\vert f\left( x,y\right) \right\vert ^{q}d\mu _{2}\left(
y\right) \right) ^{p/q}d\mu _{1}\left( x\right) \right) ^{1/p},\ \ \ f\in
L^{0}\left( G\right) ,
\end{equation*}%
and let 
\begin{equation*}
E_{p\times q}=\left\{ f\in L^{0}\left( G\right) :\left\Vert f\right\Vert
_{p\times q}<\infty \right\} .
\end{equation*}%
Equipped with the norm $\left\Vert \cdot \right\Vert _{p\times q}$ the space 
$E_{p\times q}$ is a translation invariant B.f.s. over $G$ with the Fatou
property and o.c.-norm. It is readily verified that $E_{p\times q}$ is \emph{%
not} rearrangement invariant whenever $p\neq q$.
\end{example}

The following observation may be intuitively clear.

\begin{lemma}
\label{Lem1104}Let $E\neq \left\{ 0\right\} $ be a translation invariant
B.f.s. over $G$. Then the carrier of $E$ is $G$.
\end{lemma}

\begin{proof}
Let $B\in \mathcal{B}\left( G\right) $ with $\mu \left( B\right) >0$ be
given. We have to show that there exists $C\in \mathcal{B}\left( G\right) $
satisfying $\mu \left( C\right) >0$, $C\subseteq B$ and $\chi _{C}\in E$. By
assumption there is $0<f\in E$. Hence, there exist $\varepsilon >0$ and $%
A\in \mathcal{B}\left( G\right) $ satisfying $\mu \left( A\right) >0$ and $%
0<\varepsilon \chi _{A}\leq f$. In particular, $\chi _{A}\in E$. It follows
from Theorem F in Section 59 of \cite{Ha} that there exists $y_{0}\in G$
such that $\left( \tau _{y_{0}}\chi _{A}\right) \chi _{B}>0$. Setting $%
C=\left( A+y_{0}\right) \cap B$, it follows that $\mu \left( C\right) >0$
and $\chi _{C}\leq \tau _{y_{0}}\chi _{A}\in E$, which implies that $\chi
_{C}\in E$. Hence, the set $C$ satisfies the requirements and the proof is
complete.\medskip
\end{proof}

Using the definition (see (\ref{eq01})), it is routine to verify that the K%
\"{o}the dual $E^{\times }$ is translation invariant whenever $E\neq \left\{
0\right\} $ is a translation invariant B.f.s. over $G$. If $E\neq \left\{
0\right\} $ is any rearrangement invariant B.f.s. over $G$, then $L^{\infty
}\left( G\right) \subseteq E\subseteq L^{1}\left( G\right) $ (as $\mu $ is
atomless with $\mu \left( G\right) <\infty $). The following result shows,
for any translation invariant B.f.s. $E$ over $G$, that the inclusion $%
E\subseteq L^{1}\left( G\right) $ always holds. Even if $E\neq \left\{
0\right\} $, the other inclusion $L^{\infty }\left( G\right) \subseteq E$
may fail, in general (see Example \ref{Ex1103} below).

\begin{proposition}
\label{Prop1107}Let $E$ be any translation invariant B.f.s. over $G$. Then $%
E\subseteq L^{1}\left( G\right) $ with a continuous embedding.
\end{proposition}

\begin{proof}
We can assume that $E\neq \left\{ 0\right\} $. Since then also $E^{\times
}\neq \left\{ 0\right\} $, there exists $0\leq g\in E^{\times }$ with $%
\left\Vert g\right\Vert _{E^{\times }}=1$. Given any $0\leq f\in E$, the
non-negative function $\left( x,y\right) \longmapsto f\left( x+y\right)
g\left( y\right) $, for $\left( x,y\right) \in G\times G$, is known to be
Borel measurable on $G\times G$ (\cite{Ru1}, p. 5). Since $\int_{G}fd\mu
=\int_{G}\tau _{-y}fd\mu $ for all $y\in G$ (the value $+\infty $ is
possible at this stage), we deduce from the Fubini-Tonelli theorem (for
positive functions) that 
\begin{multline*}
\int_{G}\left( \int_{G}f\left( x+y\right) g\left( y\right) d\mu \left(
y\right) \right) d\mu \left( x\right)  \\
=\int_{G}\left( \int_{G}f\left( x+y\right) d\mu \left( x\right) \right)
g\left( y\right) d\mu \left( y\right)  \\
=\left( \int_{G}fd\mu \right) \left( \int_{G}g\;d\mu \right) .
\end{multline*}%
But, for each $x\in G$, it follows from (\ref{eq02}) that 
\begin{equation*}
\int_{G}f\left( x+y\right) g\left( y\right) d\mu \left( y\right)
=\int_{G}\left( \tau _{-y}f\right) g\;d\mu \leq \left\Vert f\right\Vert
_{E}\left\Vert g\right\Vert _{E^{\times }}=\left\Vert f\right\Vert
_{E}<\infty 
\end{equation*}%
and so, by the previous identity, we have that 
\begin{equation}
\left( \int_{G}fd\mu \right) \left( \int_{G}g\;d\mu \right) \leq \mu \left(
G\right) \left\Vert f\right\Vert _{E}=\left\Vert f\right\Vert _{E}<\infty .
\label{eq1119}
\end{equation}%
Since $g>0$, its integral $\alpha =\int_{G}g\;d\mu >0$ and so $\left\Vert
f\right\Vert _{1}=\int_{G}fd\mu <\infty $, that is, $f\in L^{1}\left(
G\right) $. It follows that $E\subseteq L^{1}\left( G\right) $. As noted
earlier, the natural embedding $E\subseteq L^{1}\left( G\right) $ is
necessarily continuous. \medskip 
\end{proof}

\begin{corollary}
\label{Cor1102}Let $E\neq \left\{ 0\right\} $ be a translation invariant
B.f.s. over $G$.

\begin{enumerate}
\item[(i)] With continuous inclusions we have that $L^{\infty }\left(
G\right) \subseteq E^{\times }\subseteq L^{1}\left( G\right) $.

\item[(ii)] If, in addition, $E$ has the Fatou property, then $L^{\infty
}\left( G\right) \subseteq E\subseteq L^{1}\left( G\right) $, with
continuous inclusions.
\end{enumerate}
\end{corollary}

\begin{proof}
(i). Since $E^{\times }\neq \left\{ 0\right\} $ is also a translation
invariant B.f.s., we can apply Proposition \ref{Prop1107} to $E^{\times }$
to conclude that $E^{\times }\subseteq L^{1}\left( G\right) $ continuously.
Again by Proposition \ref{Prop1107}, now applied to $E$, we have $E\subseteq
L^{1}\left( G\right) $ continuously and hence, $L^{\infty }\left( G\right)
=L^{1}\left( G\right) ^{\times }\subseteq E^{\times }$, continuously.

(ii). Part (i), applied to $E^{\times }$, yields $L^{\infty }\left( G\right)
\subseteq E^{\times \times }\subseteq L^{1}\left( G\right) $ continuously.
If $E$ has the Fatou property, then $E=E^{\times \times }$, which completes
the proof.\medskip
\end{proof}

There exist non-trivial translation invariant B.f.s.' over $G$ for which $%
L^{\infty }\left( G\right) \nsubseteq E$.

\begin{example}
\label{Ex1103}This example is inspired by the results in \cite{Ru3}. Let $G$
be an infinite compact abelian group, equipped with normalized Haar measure $%
\mu $. Let $\mathcal{O}$ denote the collection of all open and dense subsets
of $G$. Note that $U_{1}\cap U_{2}\in \mathcal{O}$ whenever $U_{1},U_{2}\in 
\mathcal{O}$ and that $x+U\in \mathcal{O}$ whenever $U\in \mathcal{O}$ and $%
x\in G$. Define 
\begin{equation*}
J_{\mathcal{O}}=\left\{ f\in L^{\infty }\left( G\right) :\exists \ U\in 
\mathcal{O}\ \text{such that }f\chi _{U}=0\right\} .
\end{equation*}%
It is readily verified that $J_{\mathcal{O}}$ is an ideal in $L^{\infty
}\left( G\right) $, which is also translation invariant (that is, $\tau
_{x}f\in J_{\mathcal{O}}$ whenever $f\in J_{\mathcal{O}}$ and $x\in G$).
Furthermore, $J_{\mathcal{O}}\neq \left\{ 0\right\} $. Indeed, there exists $%
U\in \mathcal{O}$ with $\mu \left( U\right) <1$ (actually, for each $%
\varepsilon >0$ there exists $V\in \mathcal{O}$ such that $\mu \left(
V\right) <\varepsilon $; see Theorem 2.4 in \cite{Ru3}), in which case $%
0\neq \chi _{G\backslash U}\in J_{\mathcal{O}}$. Let $E_{\mathcal{O}}$
denote the norm closure of $J_{\mathcal{O}}$ in $L^{\infty }\left( G\right) $%
. Since the norm closure of any ideal is again an ideal (cf. Theorem 100.11
in \cite{Z}), it follows that $E_{\mathcal{O}}$ is an ideal in $L^{\infty
}\left( G\right) $ and it is easily verified that $E_{\mathcal{O}}$ is
translation invariant. Hence, $\left( E_{\mathcal{O}},\left\Vert \cdot
\right\Vert _{\infty }\right) $ is a non-zero translation invariant B.f.s.
over $G$.

The claim is that $\chi _{G}\notin E_{\mathcal{O}}$ (and hence, that $%
L^{\infty }\left( G\right) \nsubseteq E_{\mathcal{O}}$). Actually, $E_{%
\mathcal{O}}\cap C\left( G\right) =\left\{ 0\right\} $. Indeed, suppose that 
$0\neq f\in E_{\mathcal{O}}\cap C\left( G\right) $. Then there exist an $%
\varepsilon >0$ and a non-empty open set $W\subseteq G$ such that $%
\left\vert f\left( x\right) \right\vert >\varepsilon $, for $x\in W$. Choose 
$g\in J_{\mathcal{O}}$ such that $\left\Vert f-g\right\Vert _{\infty
}<\varepsilon /2$, in which case $\left\vert g\left( x\right) \right\vert
>\varepsilon /2$ for all $x\in W$. But, then $g$ cannot vanish $\mu $-a.e.
on any set from $\mathcal{O}$, contradicting that $g\in J_{\mathcal{O}}$.
\end{example}

Our next aim is to show that $L^{\infty }\left( G\right) \subseteq E$
whenever $E\neq \left\{ 0\right\} $ and $E$ has o.c.-norm (see Proposition %
\ref{Prop1108}). First we require some preliminary results.

\begin{lemma}
\label{Lem1105}Let $E\neq \left\{ 0\right\} $ be a translation invariant
B.f.s. over $G$ with o.c.-norm. For each set $A\in \Sigma $ with $\chi
_{A}\in E$ the following assertions hold.

\begin{enumerate}
\item[(i)] For each $\varepsilon >0$ there exists $\delta >0$ such that $%
\left\Vert \chi _{B}\right\Vert _{E}\leq \varepsilon $ whenever $B\in 
\mathcal{B}\left( G\right) $ satisfies $B\subseteq A$ and $\mu \left(
B\right) \leq \delta $.

\item[(ii)] If a sequence $\left( A_{n}\right) _{n=1}^{\infty }$ in $%
\mathcal{B}\left( G\right) $ satisfies $A_{n}\downarrow \emptyset $, then 
\begin{equation*}
\lim_{n\rightarrow \infty }\sup_{y\in G}\left\Vert \chi _{A_{n}}\left( \tau
_{y}\chi _{A}\right) \right\Vert _{E}=0.
\end{equation*}

\item[(iii)] For each $B\in \mathcal{B}\left( G\right) $ we have that $%
\lim_{y\rightarrow 0}\left\Vert \left( \chi _{B}-\tau _{y}\chi _{B}\right)
\chi _{A}\right\Vert _{E}=0$.

\item[(iv)] $\lim_{y\rightarrow 0}\left\Vert \chi _{A}-\tau _{y}\chi
_{A}\right\Vert _{E}=0$.
\end{enumerate}
\end{lemma}

\begin{proof}
(i). If the statement does not hold, then there exist $\varepsilon >0$ and a
sequence $\left( B_{n}\right) _{n=1}^{\infty }$ in $\mathcal{B}\left(
G\right) $ with $B_{n}\subseteq A$, $\mu \left( B_{n}\right) \leq 2^{-n}$
and $\left\Vert \chi _{B_{n}}\right\Vert _{E}\geq \varepsilon $ for all $%
n\in \mathbb{N}$. The sets $C_{n}=\bigcup\nolimits_{k=n}^{\infty }B_{k}$
satisfy $\mu \left( C_{n}\right) \leq 2^{-n+1}$ for all $n\in \mathbb{N}$
and $C_{n}\downarrow _{n}$. Hence, $\chi _{A}\geq \chi _{C_{n}}\downarrow 0$
in $E$ and so, by order continuity of the norm, we have $\left\Vert \chi
_{C_{n}}\right\Vert _{E}\downarrow 0$. But, $\chi _{C_{n}}\geq \chi _{B_{n}}$
implies that $\left\Vert \chi _{C_{n}}\right\Vert _{E}\geq \left\Vert \chi
_{B_{n}}\right\Vert _{E}\geq \varepsilon $ for all $n$, which is a
contradiction.

(ii). Given $\varepsilon >0$, by part (i) there exists $\delta >0$ such that 
$\left\Vert \chi _{B}\right\Vert _{E}\leq \varepsilon $ whenever $B\in 
\mathcal{B}\left( G\right) $ satisfies $B\subseteq A$ and $\mu \left(
B\right) \leq \delta $. Since $A_{n}\downarrow \emptyset $, there exists $%
N\in \mathbb{N}$ such that $\mu \left( A_{n}\right) \leq \delta $ for all $%
n\geq N$. Observe that 
\begin{equation*}
\left\Vert \chi _{A_{n}}\left( \tau _{y}\chi _{A}\right) \right\Vert
_{E}=\left\Vert \tau _{-y}\left( \chi _{A_{n}}\tau _{y}\chi _{A}\right)
\right\Vert _{E}=\left\Vert \left( \tau _{-y}\chi _{A_{n}}\right) \chi
_{A}\right\Vert _{E},\ \ n\in \mathbb{N},\ \ \ y\in G.
\end{equation*}%
Since $\left( \tau _{-y}\chi _{A_{n}}\right) \chi _{A}=\chi _{\left(
A_{n}-y\right) \cap A}$ and 
\begin{equation*}
\mu \left( \left( A_{n}-y\right) \cap A\right) \leq \mu \left(
A_{n}-y\right) =\mu \left( A_{n}\right) \leq \delta ,\ \ \ n\geq N,
\end{equation*}%
it follows that $\left\Vert \left( \tau _{-y}\chi _{A_{n}}\right) \chi
_{A}\right\Vert _{E}\leq \varepsilon $ for all $y\in G$ and $n\geq N$.
Hence, 
\begin{equation*}
\sup_{y\in G}\left\Vert \left( \tau _{-y}\chi _{A_{n}}\right) \chi
_{A}\right\Vert _{E}\leq \varepsilon ,\ \ \ n\geq N,
\end{equation*}%
which completes the proof of part (ii).

(iii). Fix $B\in \mathcal{B}\left( G\right) $. Note that $\left\vert \chi
_{B}-\tau _{y}\chi _{B}\right\vert =\chi _{\left( B+y\right) \Delta B}$, for 
$y\in G$, where $\Delta $ denotes the symmetric difference. It follows from
the continuity of translations in $L^{1}\left( G\right) $ (see \cite{Ru1},
p.3) that $\lim_{y\rightarrow 0}\left\Vert \chi _{B}-\tau _{y}\chi
_{B}\right\Vert _{1}=0$ and hence, that $\lim_{y\rightarrow 0}\mu \left(
\left( B+y\right) \Delta B\right) =0$. Given $\varepsilon >0$, it follows
from part (i) that there exists $\delta >0$ such that $\left\Vert \chi
_{C}\right\Vert _{E}\leq \varepsilon $ whenever $C\in \mathcal{B}\left(
G\right) $ satisfies $C\subseteq A$ and $\mu \left( C\right) \leq \delta $.
If $U$ is a neighbourhood of $0\in G$ such that $\mu \left( \left(
B+y\right) \Delta B\right) \leq \delta $ for all $y\in U$, then $\mu \left(
A\cap \left( \left( B+y\right) \Delta B\right) \right) \leq \delta $ and so, 
\begin{equation*}
\left\Vert \left( \chi _{B}-\tau _{y}\chi _{B}\right) \chi _{A}\right\Vert
_{E}=\left\Vert \chi _{A\cap \left( \left( B+y\right) \Delta B\right)
}\right\Vert _{E}\leq \varepsilon ,\ \ \ y\in U.
\end{equation*}%
This suffices for the proof of part (iii).

(iv). Since the carrier of $E$ is $G$ (cf. Lemma \ref{Lem1104}), there
exists a sequence $\left( H_{n}\right) _{n=1}^{\infty }$ in $\mathcal{B}%
\left( G\right) $ with $\chi _{H_{n}}\in E$ for all $n$ and $H_{n}\uparrow G$%
, i.e., $G\backslash H_{n}\downarrow \emptyset $. Given $\varepsilon >0$, it
follows from part (ii) that there is $N\in \mathbb{N}$ such that $\sup_{y\in
G}\left\Vert \chi _{G\backslash H_{N}}\left( \tau _{y}\chi _{A}\right)
\right\Vert _{E}\leq \varepsilon /3$. By part (iii), there is a
neighbourhood $U$ of $0\in G$ such that $\left\Vert \chi _{H_{N}}\left( \chi
_{A}-\tau _{y}\chi _{A}\right) \right\Vert _{E}\leq \varepsilon /3$ for all $%
y\in U$. Consequently, for each $y\in U$ we have 
\begin{eqnarray*}
\left\Vert \chi _{A}-\tau _{y}\chi _{A}\right\Vert _{E} &\leq &\left\Vert
\chi _{H_{N}}\left( \chi _{A}-\tau _{y}\chi _{A}\right) \right\Vert
_{E}+\left\Vert \chi _{G\backslash H_{N}}\left( \chi _{A}-\tau _{y}\chi
_{A}\right) \right\Vert _{E} \\
&\leq &\varepsilon /3+\left\Vert \chi _{G\backslash H_{N}}\left( \tau
_{0}\chi _{A}\right) \right\Vert _{E}+\left\Vert \chi _{G\backslash
H_{N}}\left( \tau _{y}\chi _{A}\right) \right\Vert _{E}\leq \varepsilon .
\end{eqnarray*}%
This completes the proof of part (iv).\medskip
\end{proof}

We can now establish the result alluded to above.

\begin{proposition}
\label{Prop1108}Let $E\neq \left\{ 0\right\} $ be a translation invariant
B.f.s. over $G$ with o.c.-norm. The following statements hold.

\begin{enumerate}
\item[(i)] For each $f\in E$ it is the case that $\lim_{y\rightarrow
0}\left\Vert f-\tau _{y}f\right\Vert _{E}=0$.

\item[(ii)] $L^{\infty }\left( G\right) \subseteq E\subseteq L^{1}\left(
G\right) $ with continuous inclusions.

\item[(iii)] $C\left( G\right) $ is a dense subspace of $E$.
\end{enumerate}
\end{proposition}

\begin{proof}
(i). Denote by $\limfunc{sim}\left( \mathcal{B}\left( G\right) \right) $ the
space of all simple functions based on $\mathcal{B}\left( G\right) $. Claim: 
$\limfunc{sim}\left( \mathcal{B}\left( G\right) \right) \cap E$ is dense in $%
E$. Indeed, given $0\leq f\in E$, there exists a sequence $\left(
s_{n}\right) _{n=1}^{\infty }$ in $\limfunc{sim}\left( \mathcal{B}\left(
G\right) \right) ^{+}$ such that $0\leq s_{n}\uparrow _{n}f$ $\mu $-a.e. on $%
G$. Since $E$ is an ideal in $L^{0}\left( G\right) $, it is clear that $%
s_{n}\in \limfunc{sim}\left( \mathcal{B}\left( G\right) \right) \cap E$, for 
$n\in \mathbb{N}$. The order continuity of the norm in $E$ implies that $%
\left\Vert f-s_{n}\right\Vert _{E}\rightarrow 0$ as $n\rightarrow \infty $.
This suffices for the proof of the claim.

Let $s\in \limfunc{sim}\left( \mathcal{B}\left( G\right) \right) \cap E$.
Then $s=\sum_{j=1}^{k}\alpha _{j}\chi _{A_{j}}$, where $A_{j}\in \mathcal{B}%
\left( G\right) $ with $\chi _{A_{j}}\in E$ and $\alpha _{j}\in \mathbb{C}$,
for $j=1,\ldots ,n$. Therefore, Lemma \ref{Lem1105} (iv) implies that $%
\left\Vert s-\tau _{y}s\right\Vert _{E}\rightarrow 0$ as $y\rightarrow 0$ in 
$G$. Since $\limfunc{sim}\left( \mathcal{B}\left( G\right) \right) \cap E$
is dense in $E$, the result of (i) is now clear.

(ii). That $E\subseteq L^{1}\left( G\right) $ has been shown in Proposition %
\ref{Prop1107}. For the inclusion of $L^{\infty }\left( G\right) $ into $E$,
we begin with a general observation, which is of independent interest (only
special cases will be used in the proofs of (ii) and (iii)).

Let $f\in E$ and $\lambda \in M\left( G\right) $. Since $f\in L^{1}\left(
G\right) $ (cf. Proposition \ref{Prop1107}), the convolution $f\ast \lambda
\in L^{1}\left( G\right) $ exists (see e.g. \cite{Ru1}, Section 1.3.2). The
claim is that $f\ast \lambda \in E$. Indeed, it follows from part (i) that
the function $F:y\longmapsto \tau _{y}f$, for $y\in G$, is continuous from $%
G $ into $E$. So the range of $F$ is a compact metric space and hence, is
separable. Via the Pettis measurability theorem (\cite{DU}, p. 42) it
follows that $F$ is strongly $\left\vert \lambda \right\vert $-measurable
and bounded. Consequently, the Bochner integral $\int_{G}^{\left( B\right)
}Fd\lambda =\int_{G}^{\left( B\right) }\tau _{y}fd\lambda \left( y\right)
\in E\subseteq L^{1}\left( G\right) $ exists. Using \cite{DU}, Theorem
II.2.6 (for continuous functionals on $E$) and Fubini's theorem, we find for
all $g\in L^{\infty }\left( G\right) \subseteq E^{\times }\subseteq E^{\ast
} $ that 
\begin{eqnarray*}
\left\langle \int_{G}^{\left( B\right) }\tau _{y}fd\lambda \left( y\right)
,g\right\rangle &=&\int_{G}\left\langle \tau _{y}f,g\right\rangle d\lambda
\left( y\right) \\
&=&\int_{G}\left( \int_{G}\left( \tau _{y}f\right) \left( x\right) g\left(
x\right) d\mu \left( x\right) \right) d\lambda \left( y\right) \\
&=&\int_{G}\left( \int_{G}f\left( x-y\right) d\lambda \left( y\right)
\right) g\left( x\right) d\mu \left( x\right) =\left\langle f\ast \lambda
,g\right\rangle .
\end{eqnarray*}%
We can conclude that 
\begin{equation}
f\ast \lambda =\int_{G}^{\left( B\right) }\tau _{y}fd\lambda \left( y\right)
\in E,\ \ \ f\in E,\ \ \lambda \in M\left( G\right) .  \label{eq03}
\end{equation}%
This proves the claim.

Fix any $0<f_{0}\in E$. By what has just been proved, we have that $%
f_{0}\ast \mu \in E$. But, $f_{0}\ast \mu =\left( \int_{G}f_{0}d\mu \right)
\chi _{G}$ and so we can conclude that $\chi _{G}\in E$ and hence, that $%
L^{\infty }\left( G\right) \subseteq E$. Furthermore, if $f\in L^{\infty
}\left( G\right) $, then $\left\vert f\right\vert \leq \left\Vert
f\right\Vert _{\infty }\chi _{G}$ implies that $\left\Vert f\right\Vert
_{E}\leq \left\Vert \chi _{G}\right\Vert _{E}\left\Vert f\right\Vert
_{\infty }$. Consequently, the embedding $L^{\infty }\left( G\right)
\subseteq E$ is continuous.

(iii). Let $f\in E$ be fixed and $\varepsilon >0$ be given. It follows from
(i) that there exists an open neighbourhood $U$ of $0$ in $G$ such that $%
\left\Vert f-\tau _{y}f\right\Vert _{E}\leq \varepsilon $ for all $y\in U$.
Defining $h=\mu \left( U\right) ^{-1}\chi _{U}$, we have $f\ast h\in C\left(
G\right) $ (as $f\in L^{1}\left( G\right) $ and $h\in L^{\infty }\left(
G\right) $; cf. \cite{Ru1}, Section 1.1.6). Furthermore, (\ref{eq03})
implies that 
\begin{equation*}
f-\left( f\ast h\right) =\mu \left( U\right) ^{-1}\int_{U}^{\left( B\right)
}\left( f-\tau _{y}f\right) d\mu \left( y\right) ,
\end{equation*}%
which implies that 
\begin{equation*}
\left\Vert f-\left( f\ast h\right) \right\Vert _{E}\leq \mu \left( U\right)
^{-1}\int_{U}\left\Vert f-\tau _{y}f\right\Vert _{E}d\mu \left( y\right)
\leq \varepsilon ;
\end{equation*}%
cf. \cite{DU}, Theorem II.2.4 (ii). The proof is thereby complete. \medskip
\end{proof}

\begin{remark}
\label{Rem01}

\begin{enumerate}
\item[(a)] Let $E\neq \left\{ 0\right\} $ be a translation invariant B.f.s.
over $G$ with o.c.-norm. Then $f\ast \lambda \in E$ for all $f\in E$ and $%
\lambda \in M\left( G\right) $; see the proof of part (ii) in the above
proposition. Moreover, it follows from (\ref{eq03}) that $\left\Vert f\ast
\lambda \right\Vert _{E}\leq \left\Vert \lambda \right\Vert _{M\left(
G\right) }\left\Vert f\right\Vert _{E}$. Consequently, for each $\lambda \in
M\left( G\right) $, the operator $T_{\lambda }^{E}:E\rightarrow E$ of
convolution with $\lambda $ exists and satisfies $\left\Vert T_{\lambda
}^{E}\right\Vert \leq \left\Vert \lambda \right\Vert _{M\left( G\right) }$.

It can be shown that a similar result holds for translation invariant
B.f.s.' over $G$ which have the Fatou property. However, for general
translation invariant B.f.s.' over $G$ this result fails. In fact, for the
translation invariant B.f.s. $E_{\mathcal{O}}$ of Example \ref{Ex1103} it
can be shown that the only measures $\lambda \in M\left( G\right) $ for
which the convolution operator $T_{\lambda }^{E_{\mathcal{O}}}:E_{\mathcal{O}%
}\rightarrow E_{\mathcal{O}}$ exists, are the discrete measures. See \cite%
{OM}.

\item[(b)] A translation invariant B.f.s. $E$ over $G$ with the property
that \linebreak $\left\Vert f-\tau _{y}f\right\Vert _{E}\rightarrow 0$ as $%
y\rightarrow 0$ in $G$, for every $f\in E$, necessarily has o.c.-norm. The
proof of this fact is somewhat involved and is not needed in this note. See 
\cite{OM}. This fact also implies that any translation invariant B.f.s. over 
$G$ in which $C\left( G\right) $ is dense, necessarily has o.c.-norm.
\end{enumerate}
\end{remark}

Next we discuss the phenomenon that a general translation invariant B.f.s.
need not be reflection invariant (see the definition below). For any $f\in
L^{0}\left( G\right) $, its reflection $\tilde{f}\in L^{0}\left( G\right) $
is defined by 
\begin{equation*}
\tilde{f}\left( x\right) =f\left( -x\right) ,\ \ \ x\in G.
\end{equation*}

\begin{definition}
\label{Def1102}A B.f.s. $E$ over $G$ is called \emph{reflection invariant}
if it has the property that $\tilde{f}\in E$ with $\left\Vert \tilde{f}%
\right\Vert _{E}=\left\Vert f\right\Vert _{E}$ whenever $f\in E$.
\end{definition}

Evidently, every rearrangement invariant B.f.s. over $G$ is reflection
invariant. Note that the B.f.s. $E_{p\times q}$ of Example \ref{Ex02}, with $%
p\neq q$, is translation and reflection invariant but, it is not
rearrangement invariant. Furthermore, it is routine to verify, for any
reflection invariant B.f.s. $E$ over $G$ with carrier equal to $G$, that the
K\"{o}the dual $E^{\times }$ is also reflection invariant. However, a
translation invariant B.f.s. over $G$ need not be reflection invariant.

\begin{example}
\label{Ex01}Let $0<g\in L^{1}\left( G\right) $ be fixed. Define 
\begin{equation}
\left\Vert f\right\Vert _{E_{g}}=\sup_{y\in G}\int_{G}\left\vert
f\right\vert \left( \tau _{y}g\right) d\mu ,\ \ \ f\in L^{0}\left( G\right) ,
\label{eqRI02}
\end{equation}%
and let 
\begin{equation}
E_{g}=\left\{ f\in L^{0}\left( G\right) :\left\Vert f\right\Vert
_{E_{g}}<\infty \right\} ,  \label{eqRI01}
\end{equation}%
equipped with the norm $\left\Vert \cdot \right\Vert _{E_{g}}$. It is
readily verified that $E_{g}$ is a translation invariant B.f.s. over $G$
with the Fatou property. Actually, $E_{g}$ is the largest translation
invariant B.f.s. $F$ over $G$ with the property that $g\in F^{\times }$.

Consider the group $G=\mathbb{T}$ with normalized Haar measure (that is, $%
d\mu =\frac{1}{2\pi }dx$ on $\left[ -\pi ,\pi \right] $, where $dx$ denotes
Lebesgue measure). Functions on $\mathbb{T}$ will be identified with $2\pi $%
-periodic functions on $\mathbb{R}$. The $2\pi $-periodic function $g:%
\mathbb{R\rightarrow R}$ is defined by setting 
\begin{equation*}
g\left( x\right) =\left\{ 
\begin{array}{cc}
0 & \text{if }-\pi <x\leq 0 \\ 
\frac{1}{\sqrt{x}} & \text{if }0<x\leq \pi%
\end{array}%
\right. .
\end{equation*}%
Define the function norm $\left\Vert \cdot \right\Vert _{E_{g}}$ on $%
L^{0}\left( G\right) $ by (\ref{eqRI02}) and let $E_{g}$ be the translation
invariant B.f.s. defined by (\ref{eqRI01}). It is clear that the function $%
g\notin E_{g}$. However, a direct computation shows that $\tilde{g}\in E_{g}$%
. Therefore, $E_{g}$ is not reflection invariant. It can be shown that the
space $E_{g}$ does not have o.c.-norm. There also exist translation
invariant B.f.s.' over $\mathbb{T}$ which are not reflection invariant but,
have both the Fatou property and o.c.-norm.
\end{example}

\section{Fourier multiplier functions}

As before, $G$ is an infinite compact abelian group equipped with normalized
Haar measure $\mu $. Let $\Gamma $ be the dual group of $G$; see, for
example, Section 1.2 of \cite{Ru1}. For $x\in G$ and $\gamma \in \Gamma $ we
write $\gamma \left( x\right) =\left( x,\gamma \right) $. Since we
frequently consider $\Gamma $, as well as its linear span $\tau \left(
G\right) $, as subsets of $C\left( G\right) $, we shall write the group
operation in $\Gamma $ as \textit{multiplication}, that is, if $\gamma
_{1},\gamma _{2}\in \Gamma $, then $\gamma _{1}\gamma _{2}\in \Gamma $ is
given by 
\begin{equation*}
\left( x,\gamma _{1}\gamma _{2}\right) =\left( x,\gamma _{1}\right) \left(
x,\gamma _{2}\right) ,\ \ \ x\in G.
\end{equation*}%
The identity element of $\Gamma $ is, of course, $\chi _{G}$. In particular, 
\begin{equation}
\left( x,\gamma ^{-1}\right) =\overline{\left( x,\gamma \right) }=\left(
-x,\gamma \right) ,\ \ \ x\in G,\ \ \gamma \in \Gamma ,  \label{eqMF12}
\end{equation}%
where the bar denotes complex conjugation. Denote by $\tau \left( G\right) $
the linear subspace of $C\left( G\right) $ consisting of all \textit{%
trigonometric polynomials} on $G$, that is, $\tau \left( G\right) $ is the
linear span of $\Gamma \subseteq C\left( G\right) $. The following fact,
which is a consequence of Proposition \ref{Prop1108} (iii), will be used.

\begin{proposition}
\label{Prop02}Let $E\neq \left\{ 0\right\} $ be a translation invariant
B.f.s. over $G$ with o.c.-norm. Then $\tau \left( G\right) $ is a dense
subspace of $E$.
\end{proposition}

\begin{proof}
The dual group $\Gamma $ separates the points of $G$ (see \cite{Ru1},
Section 1.5.2) and so it follows from the Stone-Weierstrass theorem that $%
\tau \left( G\right) $ is dense in $C\left( G\right) $ with respect to $%
\left\Vert \cdot \right\Vert _{\infty }$. By Proposition \ref{Prop1108}
(iii), $C\left( G\right) $ is dense in $E$. Since the embedding of $C\left(
G\right) $ into $E$ is continuous (as $L^{\infty }\left( G\right) \subseteq
E $ continuously), we can conclude that $\tau \left( G\right) $ is dense in $%
E$. \medskip
\end{proof}

For each $f\in L^{1}\left( G\right) $ its Fourier transform, denoted by $%
\hat{f}:\Gamma \rightarrow \mathbb{C}$, is given by 
\begin{equation*}
\hat{f}\left( \gamma \right) =\int_{G}f\left( x\right) \left( -x,\gamma
\right) d\mu \left( x\right) ,\ \ \ \gamma \in \Gamma ,
\end{equation*}%
and satisfies $\hat{f}\in c_{0}\left( \Gamma \right) $ with $\left\Vert \hat{%
f}\right\Vert _{\infty }\leq \left\Vert f\right\Vert _{1}$. Furthermore, if $%
f\in L^{1}\left( G\right) $ and $\hat{f}=0$, then $f=0$. For these facts,
see for instance Section 1.2 of \cite{Ru1}.

Let $E\neq \left\{ 0\right\} $ be translation invariant B.f.s.' over $G$.
Recall that $E$ is contained in $L^{1}\left( G\right) $ and so $\hat{f}$
exists for $f\in E$. A function $\varphi :\Gamma \rightarrow \mathbb{C}$ is
called an $E$\textit{-multiplier function} for $G$ if, for every $f\in E$,
there exists a function $g\in F$ such that $\hat{g}=\varphi \hat{f}$
(pointwise on $\Gamma $). If such a function $g$ exists, then it is
necessarily unique by the injectivity of the Fourier transform. Denote $g$
by $M_{\varphi }^{E}f$. The so defined map $M_{\varphi }^{E}:E\rightarrow E$
is linear and satisfies 
\begin{equation}
\left( M_{\varphi }^{E}f\right) ^{\wedge }=\varphi \hat{f},\ \ \ f\in E.
\label{eqMF02}
\end{equation}%
The collection of all $E$-multiplier functions for $G$ is denoted by $%
\mathcal{M}_{E}\left( G\right) $, which clearly is a (complex) vector space
of functions on $\Gamma $. For $E=L^{p}\left( G\right) $, where $1\leq p\leq
\infty $, we use the simpler notation $\mathcal{M}_{p}\left( G\right) $ in
place of $\mathcal{M}_{L^{p}\left( G\right) }\left( G\right) $. A routine
application of the closed graph theorem shows that $M_{\varphi }^{E}$ is
continuous, i.e., $M_{\varphi }^{E}\in \mathcal{L}\left( E\right) $, for
every $\varphi \in \mathcal{M}_{E}\left( G\right) $. Here, $\mathcal{L}%
\left( E\right) $ denotes the Banach space of all continuous linear
operators of $E$ into itself, equipped with the operator norm. It should be
observed that if $E$ is a translation invariant B.f.s. over $G$ with $%
L^{\infty }\left( G\right) \subseteq E$ (in which case $\gamma \in E$ for
all $\gamma \in \Gamma $) then, for all $\varphi \in \mathcal{M}_{E}\left(
G\right) $, we have 
\begin{equation}
M_{\varphi }^{E}\gamma =\varphi \left( \gamma \right) \gamma ,\ \ \ \gamma
\in \Gamma .  \label{eq04}
\end{equation}%
Indeed, given $\varphi \in \mathcal{M}_{E}\left( G\right) $ and $\gamma \in
\Gamma $, we have that $\hat{\gamma}=\chi _{\left\{ \gamma \right\} }$, from
which it follows that $\left( M_{\varphi }^{E}\gamma \right) ^{\wedge
}=\varphi \hat{\gamma}=\left( \varphi \left( \gamma \right) \gamma \right)
^{\wedge }$. This implies (\ref{eq04}). It follows from (\ref{eqMF02}) that $%
\mathcal{M}_{E}\left( G\right) $ is an algebra for pointwise multiplication
of functions on $\Gamma $.

It is shown in the following lemma that always $\mathcal{M}_{E}\left(
G\right) \subseteq \ell ^{\infty }\left( \Gamma \right) $. Recall that if $%
E\neq \left\{ 0\right\} $ is a translation invariant B.f.s. over $G$, then $%
L^{\infty }\left( G\right) \subseteq E^{\times }$; see Corollary \ref%
{Cor1102} (i). In particular, $\gamma \in E^{\times }$ for each $\gamma \in
\Gamma $. Moreover, $\left\Vert \gamma \right\Vert _{E^{\times }}=\left\Vert
\chi _{G}\right\Vert _{E^{\times }}$ for all $\gamma \in \Gamma $, because $%
\left\vert \gamma \right\vert =\chi _{G}$. Recall also that $\left\Vert
f\right\Vert _{E^{\times }}=\left\Vert f\right\Vert _{E^{\ast }}$ for each $%
f\in E^{\times }$; see the discussion following formula (\ref{eq06}).

\begin{lemma}
\label{LemMF06}Let $E\neq \left\{ 0\right\} $ be a translation invariant
B.f.s.' over $G$. Then $\mathcal{M}_{E}\left( G\right) \subseteq \ell
^{\infty }\left( \Gamma \right) $ and 
\begin{equation}
\left\Vert \varphi \right\Vert _{\infty }\leq \left\Vert M_{\varphi
}^{E}\right\Vert ,\ \ \ \varphi \in \mathcal{M}_{E}\left( G\right) .
\label{eqMF03}
\end{equation}
\end{lemma}

\begin{proof}
Fix $\varphi \in \mathcal{M}_{E}\left( G\right) $ and set $T=M_{\varphi
}^{E} $. As observed prior to the lemma, $T$ is continuous and so we can
consider the Banach space adjoint operator $T^{\ast }:E^{\ast }\rightarrow
E^{\ast }$. Let $\gamma \in \Gamma $ and consider $\gamma \in E^{\times
}\subseteq E^{\ast }$. The claim is that $T^{\ast }\gamma =\varphi \left(
\gamma ^{-1}\right) \gamma \in E^{\times }\subseteq E^{\ast }$. Indeed,
given $f\in E$ we have that 
\begin{eqnarray*}
\left\langle f,T^{\ast }\gamma \right\rangle &=&\left\langle Tf,\gamma
\right\rangle =\int_{G}\left( Tf\right) \left( x\right) \left( x,\gamma
\right) d\mu \left( x\right) \\
&=&\int_{G}\left( Tf\right) \left( x\right) \left( -x,\gamma ^{-1}\right)
d\mu \left( x\right) =\left( Tf\right) ^{\wedge }\left( \gamma ^{-1}\right)
\\
&=&\varphi \left( \gamma ^{-1}\right) \hat{f}\left( \gamma ^{-1}\right)
=\varphi \left( \gamma ^{-1}\right) \int_{G}f\left( x\right) \left(
-x,\gamma ^{-1}\right) d\mu \left( x\right) \\
&=&\varphi \left( \gamma ^{-1}\right) \int_{G}f\left( x\right) \left(
x,\gamma \right) d\mu \left( x\right) =\left\langle f,\varphi \left( \gamma
^{-1}\right) \gamma \right\rangle .
\end{eqnarray*}%
Since $E$ separates the elements of $E^{\ast }$, the claim follows. Defining
the \textit{reflected function} $\tilde{\varphi}:\Gamma \rightarrow \mathbb{C%
}$ by $\tilde{\varphi}\left( \gamma \right) =\varphi \left( \gamma
^{-1}\right) $, for $\gamma \in \Gamma $, we find that 
\begin{equation*}
\left\Vert T\right\Vert =\left\Vert T^{\ast }\right\Vert \geq \sup_{\gamma
\in \Gamma }\frac{\left\Vert T^{\ast }\gamma \right\Vert _{E^{\times }}}{%
\left\Vert \gamma \right\Vert _{E^{\times }}}=\sup_{\gamma \in \Gamma }\frac{%
\left\vert \tilde{\varphi}\left( \gamma \right) \right\vert \left\Vert
\gamma \right\Vert _{E^{\times }}}{\left\Vert \gamma \right\Vert _{E^{\times
}}}=\left\Vert \tilde{\varphi}\right\Vert _{\infty }.
\end{equation*}%
Since $\left\Vert \tilde{\varphi}\right\Vert _{\infty }=\left\Vert \varphi
\right\Vert _{\infty }$, this suffices for the proof. \medskip
\end{proof}

For the case $E=L^{p}\left( G\right) $ with $1\leq p<\infty $, the
inequality (\ref{eqMF03}) is well known; see Corollary 4.1.2 of \cite{La},
for example.

Define the family of operators 
\begin{equation}
\mathcal{M}_{E}^{\limfunc{op}}\left( G\right) =\left\{ M_{\varphi
}^{E}:\varphi \in \mathcal{M}_{E}\left( G\right) \right\} \subseteq \mathcal{%
L}\left( E\right) .  \label{eqMF13}
\end{equation}%
For each $x\in G$, a direct calculation shows that the Dirac point measure $%
\delta _{x}$ satisfies $\tau _{x}f=f\ast \delta _{x}$, for $f\in E$. Since $%
\tau _{x}\in \mathcal{L}\left( E\right) $ and $\hat{\delta}_{x}\left( \gamma
\right) =\left( -x,\gamma \right) $, for $\gamma \in \Gamma $, it follows
that $\left( x,\cdot \right) \in \mathcal{M}_{E}\left( G\right) $ for every $%
x\in G$ and $\left\{ \tau _{x}:x\in G\right\} \subseteq \mathcal{M}_{E}^{%
\limfunc{op}}\left( G\right) $. In particular, the identity operator $I=\tau
_{0}\in \mathcal{M}_{E}^{\limfunc{op}}\left( G\right) $ and the constant
function $\hat{\delta}_{0}=\mathbf{1}\in \mathcal{M}_{E}\left( G\right) $.

The following result is a consequence of Lemma \ref{LemMF06}.

\begin{corollary}
\label{CorMF05}Let $E\neq \left\{ 0\right\} $ be a translation invariant
B.f.s. over $G$. The set $\mathcal{M}_{E}\left( G\right) $ is a unital
subalgebra of $\ell ^{\infty }\left( \Gamma \right) $ and the map $\Phi :%
\mathcal{M}_{E}\left( G\right) \rightarrow \mathcal{L}\left( E\right) $,
defined by 
\begin{equation*}
\Phi \left( \varphi \right) =M_{\varphi }^{E},\ \ \ \varphi \in \mathcal{M}%
_{E}\left( G\right) ,
\end{equation*}%
is a unital algebra isomorphism onto the subalgebra $\mathcal{M}_{E}^{%
\limfunc{op}}\left( G\right) \subseteq \mathcal{L}\left( E\right) $.
Moreover, $\mathcal{M}_{E}^{\limfunc{op}}\left( G\right) $ is a unital
commutative Banach subalgebra of $\mathcal{L}\left( E\right) $.
\end{corollary}

\begin{proof}
It is routine to verify that $\mathcal{M}_{E}\left( G\right) $ is a unital
subalgebra of $\ell ^{\infty }\left( \Gamma \right) $ and that the map $\Phi
:\mathcal{M}_{E}\left( G\right) \rightarrow \mathcal{L}\left( E\right) $ is
a unital algebra homomorphism. Consequently, $\mathcal{M}_{E}^{\limfunc{op}%
}\left( G\right) =\Phi \left( \mathcal{M}_{E}\left( G\right) \right) $ is a
unital commutative subalgebra of $\mathcal{L}\left( E\right) $. The
injectivity of $\Phi $ follows from Lemma \ref{LemMF06}.

To show that $\mathcal{M}_{E}^{\limfunc{op}}\left( G\right) $ is norm closed
in $\mathcal{L}\left( E\right) $, let $\left( M_{\varphi _{n}}^{E}\right)
_{n=1}^{\infty }$ be a sequence in $\mathcal{M}_{E}^{\limfunc{op}}\left(
G\right) $ satisfying $\left\Vert S-M_{\varphi _{n}}^{E}\right\Vert
\rightarrow 0$ as $n\rightarrow \infty $ for some $S\in \mathcal{L}\left(
E\right) $. Lemma \ref{LemMF06} implies that $\left( \varphi _{n}\right)
_{n=1}^{\infty }$ is a Cauchy sequence in $\ell ^{\infty }\left( \Gamma
\right) $. So, there exists $\varphi \in \ell ^{\infty }\left( \Gamma
\right) $ such that $\left\Vert \varphi -\varphi _{n}\right\Vert _{\infty
}\rightarrow 0$ for $n\rightarrow \infty $. It follows, for a fixed $f\in
E\subseteq L^{1}\left( G\right) $, that $\hat{f}\in c_{0}\left( \Gamma
\right) $ and so 
\begin{equation}
\left( M_{\varphi _{n}}^{E}f\right) ^{\wedge }=\varphi _{n}\hat{f}%
\rightarrow \varphi \hat{f},\ \ \ n\rightarrow \infty ,  \label{eqMF18}
\end{equation}%
in $\ell ^{\infty }\left( \Gamma \right) $. On the other hand, $M_{\varphi
_{n}}^{E}f\rightarrow Sf$ in $E$ for $n\rightarrow \infty $. Since $E$ is
continuously included in $L^{1}\left( G\right) $, it follows that $%
M_{\varphi _{n}}^{E}f\rightarrow Sf$ in $L^{1}\left( G\right) $. By the
continuity of the Fourier transform from $L^{1}\left( G\right) $ into $%
c_{0}\left( \Gamma \right) $ we can conclude that $\left( M_{\varphi
_{n}}^{E}f\right) ^{\wedge }\rightarrow \left( Sf\right) ^{\wedge }$ in $%
\ell ^{\infty }\left( \Gamma \right) $. Then (\ref{eqMF18}) implies that $%
\left( Sf\right) ^{\wedge }=\varphi \hat{f}$. But, $f\in E$ is arbitrary and
so $\varphi \in \mathcal{M}_{E}\left( G\right) $ with $S=M_{\varphi }^{E}$.
The proof is complete.\medskip\ 
\end{proof}

\begin{remark}

\begin{enumerate}
\item[(a)] For $E=L^{2}\left( G\right) $ it is a classical result that $%
\mathcal{M}_{2}\left( G\right) =\ell ^{\infty }\left( \Gamma \right) $,
which is an immediate consequence of the fact that the Fourier transform is
a surjective isometry from $L^{2}\left( G\right) $ onto $\ell ^{2}\left(
\Gamma \right) $. Conversely, if $E$ is a translation invariant B.f.s. over $%
G$ with $L^{\infty }\left( G\right) \subseteq E$ and $\mathcal{M}_{E}\left(
G\right) =\ell ^{\infty }\left( \Gamma \right) $, then it follows from
Proposition 1.6 of \cite{dPR} that $E=L^{2}\left( G\right) $ with equivalent
norms.

\item[(b)] Let $E\neq \left\{ 0\right\} $ be a translation invariant B.f.s.
over $G$ with \emph{o.c.-norm}. As observed in Remark \ref{Rem01}, for each $%
\lambda \in M\left( G\right) $ the operator $T_{\lambda }^{E}:E\rightarrow E$
of convolution with $\lambda $ exists. Recalling that $\left( f\ast \lambda
\right) ^{\wedge }=\hat{\lambda}\hat{f}$ for $f\in L^{1}\left( G\right) $
and $\lambda \in M\left( G\right) $, it follows that $\hat{\lambda}\in 
\mathcal{M}_{E}\left( G\right) $ with $M_{\hat{\lambda}}^{E}=T_{\lambda
}^{E} $ for all $\lambda \in M\left( G\right) $. Hence, 
\begin{equation*}
\left\{ \hat{\lambda}:\lambda \in M\left( G\right) \right\} \subseteq 
\mathcal{M}_{E}\left( G\right) .
\end{equation*}%
A similar result holds if $E$ has the Fatou property.
\end{enumerate}
\end{remark}

For any function $f\in L^{0}\left( G\right) $, the function $f^{\natural
}\in L^{0}\left( G\right) $ is defined by 
\begin{equation*}
f^{\natural }\left( x\right) =\overline{f\left( -x\right) },\ \ \ \mu \text{%
-a.e}\ \ x\in G.
\end{equation*}%
Let $E$ be a reflection and translation invariant B.f.s. over $G$. Then 
\begin{equation*}
f^{\natural }\in E\ \ \text{and\ \ }\left\Vert f^{\natural }\right\Vert
_{E}=\left\Vert f\right\Vert _{E},\ \ \ f\in E,
\end{equation*}%
and the map $f\longmapsto f^{\natural }$, for $f\in E$, is a conjugate
linear isometric involution in $E$. For each operator $T\in \mathcal{L}%
\left( E\right) $ the operator $T^{\natural }\in \mathcal{L}\left( E\right) $
is defined by 
\begin{equation*}
T^{\natural }f=\left( Tf^{\natural }\right) ^{\natural },\ \ \ f\in E,
\end{equation*}%
and satisfies $\left\Vert T^{\natural }\right\Vert =\left\Vert T\right\Vert $%
. Then $T\longmapsto T^{\sharp }$, for $T\in \mathcal{L}\left( E\right) $,
is a conjugate linear isometric involution in $\mathcal{L}\left( E\right) $.

\begin{proposition}
\label{Prop01}Let $E\neq \left\{ 0\right\} $ be a translation and reflection
invariant B.f.s. over $G$.

\begin{enumerate}
\item[(i)] If $\varphi \in \mathcal{M}_{E}\left( G\right) $, then also $\bar{%
\varphi}\in \mathcal{M}_{E}\left( G\right) $ and $M_{\bar{\varphi}%
}^{E}=\left( M_{\varphi }^{E}\right) ^{\natural }$.

\item[(ii)] The map $M_{\varphi }^{E}\longmapsto M_{\bar{\varphi}}^{E}$, for 
$\varphi \in \mathcal{M}_{E}\left( G\right) $, is a conjugate linear
isometric involution in $\mathcal{M}_{E}^{\limfunc{op}}\left( G\right) $.
\end{enumerate}
\end{proposition}

\begin{proof}
(i). Let $\varphi \in \mathcal{M}_{E}\left( G\right) $. A direct (but
careful) computation shows that 
\begin{equation*}
\left( \left( M_{\varphi }^{E}\right) ^{\natural }f\right) ^{\wedge }=\bar{%
\varphi}\hat{f},\ \ \ f\in E.
\end{equation*}%
Consequently, $\bar{\varphi}\in \mathcal{M}_{E}\left( G\right) $ and $M_{%
\bar{\varphi}}^{E}=\left( M_{\varphi }^{E}\right) ^{\natural }$.

(ii). This follows immediately from part (i) in combination with the
observations made prior to the present proposition.\medskip
\end{proof}

For $E=L^{2}\left( G\right) $ and $\varphi \in \mathcal{M}_{2}\left(
G\right) =\ell ^{\infty }\left( \Gamma \right) $, it follows (via
Plancherel's formula) that $M_{\bar{\varphi}}^{\left( 2\right) }=\left(
M_{\varphi }^{\left( 2\right) }\right) ^{\ast }$ is the Hilbert space
adjoint of $M_{\varphi }^{\left( 2\right) }$.

\section{A Fuglede type theorem}

Again let $G$ be an infinite compact abelian group with normalized Haar
measure $\mu $. As observed in Proposition \ref{Prop01} (i), if $E$ is a
translation and reflection invariant B.f.s. over $G$, then $\bar{\varphi}\in 
\mathcal{M}_{E}\left( G\right) $ whenever $\varphi \in \mathcal{M}_{E}\left(
G\right) $. For $E=L^{2}\left( G\right) $, we noted that $M_{\bar{\varphi}%
}^{\left( 2\right) }=\left( M_{\varphi }^{\left( 2\right) }\right) ^{\ast }$
is the Hilbert space adjoint of $M_{\varphi }^{\left( 2\right) }$. In
particular, $M_{\varphi }^{\left( 2\right) }$ is a normal operator on $%
L^{2}\left( G\right) $ (as $\mathcal{M}_{2}^{\limfunc{op}}\left( G\right) $
is commutative; cf. Corollary \ref{CorMF05}). Consequently, it follows from
Fuglede's theorem (see, for instance, Theorem IX.6.7 in \cite{Co}) that if $%
\varphi \in \mathcal{M}_{2}\left( G\right) =\ell ^{\infty }\left( \Gamma
\right) $ and $T\in \mathcal{L}\left( L^{2}\left( G\right) \right) $ satisfy 
$M_{\varphi }^{\left( 2\right) }T=TM_{\varphi }^{\left( 2\right) }$, then
also $M_{\bar{\varphi}}^{\left( 2\right) }T=TM_{\bar{\varphi}}^{\left(
2\right) }$. As noted in the Introduction, this latter result was extended
to the case $E=L^{p}\left( \mathbb{T}\right) $ with $1\leq p<\infty $ in
Theorem 3.1 of \cite{MR}. The purpose of this section to extend this result
to all translation and reflection invariant B.f.s.' with o.c.-norm over
arbitrary compact abelian groups (see Theorem \ref{Thm01} below).

For any $f\in L^{1}\left( G\right) $ we denote 
\begin{equation*}
\limfunc{supp}\left( \hat{f}\right) =\left\{ \gamma \in \Gamma :\hat{f}%
\left( \gamma \right) \neq 0\right\} .
\end{equation*}%
We begin with the following result.

\begin{lemma}
\label{LemFTA01}Let $E$ be a translation invariant B.f.s. over $G$ with $%
L^{\infty }\left( G\right) \subseteq E$. Let $T\in \mathcal{L}\left(
E\right) $ and $\varphi ,\psi \in \mathcal{M}_{E}\left( G\right) $.

\begin{enumerate}
\item[(i)] Suppose that $M_{\varphi }^{E}T=TM_{\psi }^{E}$. Then, for every $%
\gamma \in \Gamma $, it is the case that $\varphi \left( \xi \right) =\psi
\left( \gamma \right) $ for all $\xi \in \limfunc{supp}\left( \left( T\gamma
\right) ^{\wedge }\right) $.

\item[(ii)] Assume, in addition, that $E$ has o.c.-norm. Let $\varphi $ and $%
\psi $ have the property that, for each $\gamma \in \Gamma $, we have $%
\varphi \left( \xi \right) =\psi \left( \gamma \right) $ for all $\xi \in 
\limfunc{supp}\left( \left( T\gamma \right) ^{\wedge }\right) $. Then $%
M_{\varphi }^{E}T=TM_{\psi }^{E}$.
\end{enumerate}
\end{lemma}

\begin{proof}
(i). Fix $\gamma \in \Gamma $. Since $L^{\infty }\left( G\right) \subseteq E$%
, we have that $\gamma \in E$ and so 
\begin{equation}
M_{\varphi }^{E}T\gamma =TM_{\psi }^{E}\gamma .  \label{eqFTA03}
\end{equation}%
Recalling from (\ref{eq04}) that $M_{\psi }^{E}\gamma =\psi \left( \gamma
\right) \gamma $, it follows that $TM_{\psi }^{E}\gamma =\psi \left( \gamma
\right) T\gamma $. Hence, 
\begin{equation}
\left( TM_{\psi }^{E}\gamma \right) ^{\wedge }=\psi \left( \gamma \right)
\left( T\gamma \right) ^{\wedge }.  \label{eqFTA01}
\end{equation}%
On the other hand, 
\begin{equation}
\left( M_{\varphi }^{E}T\gamma \right) ^{\wedge }=\varphi \cdot \left(
T\gamma \right) ^{\wedge }.  \label{eqFTA02}
\end{equation}%
A combination of (\ref{eqFTA03}), (\ref{eqFTA01}) and (\ref{eqFTA02}) yields
that 
\begin{equation}
\varphi \left( \xi \right) \left( T\gamma \right) ^{\wedge }\left( \xi
\right) =\psi \left( \gamma \right) \left( T\gamma \right) ^{\wedge }\left(
\xi \right) ,  \label{eqFTA04}
\end{equation}%
for all $\xi \in \Gamma $. If $\limfunc{supp}\left( \left( T\gamma \right)
^{\wedge }\right) =\emptyset $, then there is nothing to be proved. So,
suppose there exists $\xi \in \limfunc{supp}\left( \left( T\gamma \right)
^{\wedge }\right) $. Then $\left( T\gamma \right) ^{\wedge }\left( \xi
\right) \neq 0$ and hence, (\ref{eqFTA04}) implies that $\varphi \left( \xi
\right) =\psi \left( \gamma \right) $. This proves part (i).

(ii). Fix $\gamma \in \Gamma $. If $\xi \in \Gamma \backslash \limfunc{supp}%
\left( \left( T\gamma \right) ^{\wedge }\right) $, i.e., $\left( T\gamma
\right) ^{\wedge }\left( \xi \right) =0$, then it is clear that (\ref%
{eqFTA04}) holds. On the other hand, if $\xi \in \limfunc{supp}\left( \left(
T\gamma \right) ^{\wedge }\right) $, then it follows from the hypothesis
that (\ref{eqFTA04}) holds. Hence, 
\begin{equation*}
\varphi \cdot \left( T\gamma \right) ^{\wedge }=\psi \left( \gamma \right)
\left( T\gamma \right) ^{\wedge }.
\end{equation*}%
Via (\ref{eqFTA01}) and (\ref{eqFTA02}) this yields that 
\begin{equation*}
\left( M_{\varphi }^{E}T\gamma \right) ^{\wedge }=\left( TM_{\psi
}^{E}\gamma \right) ^{\wedge }.
\end{equation*}%
The uniqueness of Fourier transforms in $L^{1}\left( G\right) $ then implies
(\ref{eqFTA03}).

Since $\gamma \in \Gamma $ is arbitrary, the linearity of $T$, $M_{\varphi
}^{E}$ and $M_{\psi }^{E}$ imply that 
\begin{equation*}
M_{\varphi }^{E}Tg=TM_{\psi }^{E}g,\ \ \ g\in \tau \left( G\right) ,
\end{equation*}%
where, as before, $\tau \left( G\right) \subseteq L^{\infty }\left( G\right)
\subseteq E$ denotes the space of all trigonometric polynomials on $G$. But, 
$E$ has o.c.-norm and so $\tau \left( G\right) $ is dense in $E$; see
Proposition \ref{Prop02}. The operators $T$, $M_{\varphi }^{E}$and $M_{\psi
}^{E}$ are continuous and hence, we can conclude that $M_{\varphi
}^{E}T=TM_{\psi }^{E}$. The proof is thereby complete. \medskip
\end{proof}

The following Fuglede type theorem is a consequence of Lemma \ref{LemFTA01}.

\begin{theorem}
\label{Thm01}Let $E\neq \left\{ 0\right\} $ be a translation and reflection
invariant B.f.s. over $G$ with o.c.-norm. Suppose that $\varphi ,\psi \in 
\mathcal{M}_{E}\left( G\right) $ and $T\in \mathcal{L}\left( E\right) $
satisfy $M_{\varphi }^{E}T=TM_{\psi }^{E}$. Then $M_{\bar{\varphi}}^{E}T=TM_{%
\bar{\psi}}^{E}$.
\end{theorem}

\begin{proof}
The reflection invariance of $E$ ensures that $\bar{\varphi},\bar{\psi}\in 
\mathcal{M}_{E}\left( G\right) $ and the order continuity of the norm in $E$
implies that $L^{\infty }\left( G\right) \subseteq E$. By Lemma \ref%
{LemFTA01} (i), the condition $M_{\varphi }^{E}T=TM_{\psi }^{E}$ implies,
for every $\gamma \in \Gamma $, that 
\begin{equation}
\varphi \left( \xi \right) =\psi \left( \gamma \right) ,\ \ \ \xi \in 
\limfunc{supp}\left( \left( T\gamma \right) ^{\wedge }\right) .
\label{eqFTA05}
\end{equation}%
It is clear that the functions $\bar{\varphi}$ and $\bar{\psi}$ (in place of 
$\varphi $ and $\psi $, respectively) also satisfy (\ref{eqFTA05}).
Therefore, via Lemma \ref{LemFTA01} (ii) applied to $\bar{\varphi},\bar{\psi}%
\in \mathcal{M}_{E}\left( G\right) $, we can conclude that $M_{\bar{\varphi}%
}^{E}T=TM_{\bar{\psi}}^{E}$. This completes the proof. \medskip
\end{proof}

The following result is a special case of Theorem \ref{Thm01}.

\begin{corollary}
Let $E\neq \left\{ 0\right\} $ be any rearrangement invariant B.f.s. over $G$
with o.c.-norm. Suppose that $\varphi ,\psi \in \mathcal{M}_{E}\left(
G\right) $ and $T\in \mathcal{L}\left( E\right) $ satisfy $M_{\varphi
}^{E}T=TM_{\psi }^{E}$. Then $M_{\bar{\varphi}}^{E}T=TM_{\bar{\psi}}^{E}$.
\end{corollary}

\noindent \texttt{Ben de Pagter}

\noindent \texttt{Delft Institute of Applied Mathematics,}

\noindent \texttt{Faculty EEMCS,}

\noindent \texttt{Delft University of Technology, }

\noindent \texttt{P.O. Box 5031, 2600 GA Delft, }

\noindent \texttt{The Netherlands.}

\noindent \texttt{b.depagter@tudelft.nl}

\bigskip

\noindent \texttt{Werner J. Ricker}

\noindent \texttt{Math.-Geogr. Fakult\"{a}t}

\noindent \texttt{Katholische Universit\"{a}t Eichst\"{a}tt-Ingolstadt}

\noindent \texttt{D-85072 Eichst\"{a}tt, GERMANY.}

\noindent \texttt{werner.ricker@ku.de}

\end{document}